\documentclass[12pt,oneside]{amsart}

\usepackage{amssymb,latexsym}

\usepackage{hyperref}
\hypersetup{
  colorlinks   = true, 
  urlcolor     = blue, 
  linkcolor    = blue, 
  citecolor   = red 
}
\usepackage{anysize}
\marginsize{4cm}{4cm}{3cm}{3cm}

\usepackage{graphicx}
\usepackage{subcaption}
\usepackage{enumitem}
\usepackage{pdfpages}

\newtheorem{theorem}{Theorem}[section]

\newtheorem{lemma}[theorem]{Lemma}

\newtheorem{corollary}[theorem]{Corollary}

\theoremstyle{definition}

\theoremstyle{remark}
\newtheorem{remark}[theorem]{Remark}
\newtheorem{question}[theorem]{Question}

\numberwithin{equation}{section}

\DeclareMathOperator{\vol}{vol}

\DeclareMathOperator{\id}{id}

\renewcommand{\epsilon}{\varepsilon}
\renewcommand{\phi}{\varphi}
\renewcommand{\kappa}{\varkappa}

\begin{document}

\title{Outer billiards of symplectically self-polar convex bodies}

\author{Mark Berezovik}

\address{Mark Berezovik, School of Mathematical Sciences, Tel Aviv University, Israel,  69978}

\email{m.berezovik@gmail.com}

\author{Misha Bialy}

\address{Misha Bialy, School of Mathematical Sciences, Tel Aviv University, Israel,  69978}

\email{bialy@tauex.tau.ac.il}

\thanks{Mark Berezovik was partially supported by the ISF grant 938/22 and partially supported by the ISF-NSFC grant 3231/23.} 
\thanks{Misha Bialy was partially supported by ISF grant 974/24.}


\date{October 2025}

\begin{abstract}
	It is known that $C^1$-smooth strictly convex Radon norms in $\mathbb{R}^2$ can be characterized by the property that the outer billiard map, which corresponds to the unit ball of the norm, has an invariant curve consisting of 4-periodic orbits.
	In higher dimensions, Radon norms are necessarily Euclidean. However, we show in this paper that the property of existence of an invariant curve of 4-periodic orbits allows a higher-dimensional extension to the class of symplectically self-polar convex bodies. Moreover, this class of convex bodies provides the first non-trivial examples of invariant hypersurfaces for outer billiard map. This is in contrast with conventional Birkhoff billiards in higher dimensions, where it was proved by Berger and Gruber that only ellipsoids have caustics. It is not known, however, if non-trivial invariant hypersurfaces can exist for higher-dimensional Birkhoff billiards.
\end{abstract}

\maketitle
\section{Introduction}
It was proved in \cite{MR1071638, MR1348796} that for higher-dimensional Birkhoff billiards only ellipsoids have caustics. Therefore, only for ellipsoids there exist invariant hypersurfaces for the billiard map which correspond to caustics.
However, these results do not prevent the existence of invariant hypersurfaces
in the phase space of the billiard map, not coming from caustics. It is an open question if such invariant hypersurfaces can exist for higher-dimensional Birkhoff billiards.
Here, we address the question on invariant hypersurfaces for multi-dimensional outer billiards \cite{tabachnikov1995dual,tabachnikov1995billiards}.

In this paper, we study a natural class of symplectically self-polar convex bodies in $\mathbb R^{2n}$. For $n=1$ this class coincides with the unit disks of Radon norms \cite{Day,Radoncurves, martini2006antinorms}.
It was observed in \cite{bialy2022self} that for $C^1$-smooth strictly convex norms in $\mathbb{R}^2$ the Radon property can be rephrased in terms of the outer billiard corresponding to the unit disk of the norm as follows. A $C^1$-smooth strictly convex norm is Radon if and only if the outer billiard for the unit disk of the norm has an  invariant curve consisting of 4-periodic orbits.

In higher dimensions, Radon norms are necessarily Euclidean \cite{Day}. However, we show in this paper that the property of existence of invariant curve consisting of $4$-periodic orbits allows a higher-dimensional extension to the class of symplectically self-polar convex bodies.

We shall prove in Theorem \ref{thm:main} and Section~\ref{sec:invariant_hypersurface}, that for symplectically self-polar convex body of class $C^1$, the outer billiard has an invariant hypersurface consisting of centrally symmetric 4-periodic orbits. Moreover, Theorem~\ref{thm:main_reverse} states that the converse is also true for $C^2$-smooth strictly convex bodies (it is not clear to us if the result still holds for the $C^1$ case).

Let $X \subset \mathbb{R}^{2n}$ be a convex body with the origin in its interior, and let $\omega$ be a standard linear symplectic form in $\mathbb{R}^{2n}$. Using this, we define the symplectic polar convex body
\[
	X^\omega = \{y \in \mathbb{R}^{2n}:\ \forall x \in X\ \omega(x,y) \leq 1\}.
\]
This definition is similar to the definition of the Euclidean polar convex body
\[
	X^\circ = \{y \in \mathbb{R}^{2n}:\ \forall x \in X\ \langle x,y\rangle \leq 1\}
\]
where $\langle \cdot, \cdot \rangle$ denotes the standard inner product in $\mathbb{R}^{2n}$. The relation between $X^\omega$ and $X^\circ$ is given by $X^\omega = JX^\circ$, where $J$ is the multiplication by $\sqrt{-1}$ under the standard identification of $\mathbb{C}^n \cong \mathbb{R}^{2n}$.

If $X = X^\omega$, we say that $X$ is a \emph{symplectically self-polar convex body}. It is important to note that every symplectically self-polar convex body is centrally symmetric~\cite[Lemma 3.1]{berezovik2022symplectic}. Boundaries of such bodies in the plane are known as Radon curves~\cite{Radoncurves,martini2006antinorms}. The class of symplectically self-polar convex bodies was previously studied in~\cite{berezovik2022symplectic,berezovik2023symplectically} in the context of volume and capacity. In particular, it was shown that Mahler's conjecture~\cite{Mahler1939}
\[
	\vol K \cdot \vol K^\circ \geq \frac{4^n}{n!}
\]
for any centrally symmetric convex body $K \subset \mathbb{R}^n$ is equivalent to the conjecture
\[
	\vol X \geq \frac{2^n}{n!}
\]
for any symplectically self-polar convex body $X \subset \mathbb{R}^{2n}$.

\subsection*{Acknowledgments.} The authors thank Roman Karasev and the anonymous referee for useful remarks and comments.

\section{Preliminaries}

In this work, we discuss some properties of symplectic outer billiard maps that are generated by symplectically self-polar convex bodies with smooth boundary. The symplectic outer billiard map is well-defined if the body has $C^1$-smooth boundary and is strictly convex. In the case of symplectically self-polar bodies, strict convexity is equivalent to smoothness. This is stated in the next lemma.

\begin{lemma}\label{lem:strict_conv}
	Let $X \subset \mathbb{R}^{2n}$ be a symplectically self-polar convex body. Then $X$ is strictly convex if and only if $X$ has $C^1$-smooth boundary.
\end{lemma}

\begin{proof}
	In convex geometry, it is well known that the body $X^\circ$ has a $C^1$-smooth boundary if and only if $X$ is strictly convex. In our case, we have $X = X^\omega = JX^\circ$. Therefore, $X$ and $X^\circ$ coincide up to rotation. Hence, $X$ is strictly convex if and only if $X$ has a $C^1$-smooth boundary.
\end{proof}

We now turn to the crucial notation of this paper. Let $X \subset \mathbb{R}^{2n}$ be an arbitrary convex body with the origin in its interior and $\partial X$ be $C^1$-smooth. Let us introduce a map $f \colon \partial X \to \partial X^\omega$ which is uniquely determined by the requirement $\omega(x,f(x))= 1$.

Observe that the characteristic line of $\partial X$ at $x$ (which is defined as $\ker\omega|_{T_x \partial X} \subseteq T_x \partial X$)  is parallel to $f(x)$ (this fact was noticed before in~\cite[Proof of Lemma 3.6]{karasev2024convex}). Indeed, from the definition of $X^\omega$, we have $\omega(y,f(x)) \leq 1$ for every $y \in X$. Hence, the hyperplane $\omega(\cdot,f(x)) = 1$ is a tangent hyperplane of $\partial X$ at the point $x$. Then it is easy to see that $f(x)$ is actually parallel to the characteristic line in this hyperplane. Moreover, one can show that $f$ is just a restriction of a Hamiltonian vector field generated by the Hamiltonian function $H = \|\cdot\|_X$ on $\partial X$. Here $\|\cdot\|_X$ is a norm on $\mathbb{R}^{2n}$ whose unit ball is $X$. Precisely, for $x \in \mathbb{R}^{2n}$,
\[
	\|x\|_X = \inf \{\lambda > 0 : \lambda^{-1} x \in X \}.
\]
Therefore, if $X$ has a $C^k$-smooth boundary $\partial X$, for $k \geq 1$, then $f$ is $C^{k-1}$-smooth map, as a map from $\partial X$ to $\mathbb{R}^{2n}$.

For a symplectically self-polar convex body $X \subset \mathbb{R}^{2n}$ the map $f$ acts from $\partial X$ to itself. Therefore, we can consider its composition with itself. Although the next statement is simple, we state it explicitly, since it is used throughout the whole paper.

\begin{lemma} \label{lem:f_compose}
	Let $X \subset \mathbb{R}^{2n}$ be a symplectically self-polar convex body with a $C^1$-smooth boundary. Then $f \circ f = -\id_{\partial X}$.
\end{lemma} 

\begin{proof}
	Let $x \in \partial X$, then by the definition of $f$ we have $\omega(x,f(x)) = 1$. Denote $y = f(x) \in \partial X$, then $\omega(y, -x) = \omega(x,f(x)) = 1$. Since $X$ is centrally symmetric, it follows that $-x \in \partial X = \partial X^\omega$. By the definition of $f$ this means that $f(y) = -x$.
\end{proof}

In addition to Lemma~\ref{lem:strict_conv}, we find that symplectically self-polar bodies with $C^2$-smooth boundaries have positive curvature everywhere.
\begin{lemma}\label{lem:strong_conv}
	Let $X \subset \mathbb{R}^{2n}$ be a symplectically self-polar convex body with $C^2$-smooth boundary, then $\partial X$ has positive curvature at every point (the second fundamental form is positive definite). Moreover, for every point $x \in \partial X$ and non-zero tangent vector $\xi \in T_x \partial X$ we have $\omega(\xi,\nabla_{\xi}f) > 0$, where $\nabla$ is the standard connection in $\mathbb{R}^{2n}$.
\end{lemma}

\begin{proof}
	Observe that $f \colon \partial X \to \partial X$ is a $C^1$-smooth map and $f^2 = -\id_{\partial X}$. Therefore, $f$ is a $C^1$-diffeomorphism of $\partial X$.

	Assume, for the sake of contradiction, that there exists a point $x^* \in \partial X$ and a non-zero vector $\xi \in T_{x^*} \partial X$ such that $\nabla_{\xi} n = 0$, where $n(x)$ is the unit outward-pointing normal vector field on $\partial X$.

	We write $f(x) = \lambda(x) \cdot Jn(x)$ for a positive $C^1$-smooth function $\lambda$ on $\partial X$. Define the function $g \colon \partial X \to \mathbb{R}$ by $g(x) = \omega(x^*,f(x))$. The function $g$ attains its maximum at the point $x^*$, giving us $\xi(g) = 0$, where $\xi(g)$ is the derivative of $g$ with respect to $\xi$. At the same time, we have $$\xi(g) = \xi(\lambda) \cdot \omega(x^* , Jn(x^*)) + \lambda(x^*) \cdot \omega(x^*,J\nabla_{\xi}n) = \xi(\lambda) \cdot \omega(x^* , Jn(x^*)).$$ Therefore, $\xi(\lambda) = 0$.  Hence $f_*(\xi) =0$, since $$ f_*(\xi) = \nabla_\xi(f) = \xi(\lambda) \cdot Jn(x^*) + \lambda(x^*)\cdot J\nabla_\xi n,$$ where $f_*$ is the differential of $f$. 
	This contradicts $f$ being a diffeomorphism. Therefore, $\partial X$ has positive curvature at every point.

	Now consider a point $x \in \partial X$ and a non-zero tangent vector $\xi \in T_x \partial X$, then we get from the previous formula $$\omega(\xi,\nabla_{\xi}f) = \xi(\lambda)\cdot \omega(\xi,Jn(x)) + \lambda(x) \cdot \omega(\xi,J\nabla_{\xi}n).$$ Notice that $\omega(\xi,Jn(x)) = 0$, since $Jn(x)$ is the characteristic direction at the point $x$. Therefore, $\omega(\xi,\nabla_{\xi}f) = \lambda(x) \langle\xi,\nabla_{\xi}n\rangle > 0$, since $\partial X$ has positive curvature at every point.
\end{proof}

\section{4-Periodic trajectories}

Recall the definition of the symplectic outer billiard map~\cite{tabachnikov1995dual,tabachnikov1995billiards}. Consider a strictly convex body $X \subset \mathbb{R}^{2n}$ with a $C^1$-smooth boundary. For every point $z \in \mathbb{R}^{2n} \setminus X$ there exists a unique point $x \in \partial X$ such that the line $zx$ is the tangent characteristic line at the point $x$ and $\omega(x,x-z) > 0$ (orientation of the characteristic line bundle). The symplectic outer billiard map $T \colon \mathbb{R}^{2n} \setminus X \to \mathbb{R}^{2n} \setminus X$ at $z$ is defined as $T(z) = z + 2\cdot(x-z)$.

\begin{figure}[ht]
	\centering
	\includegraphics[width=0.8\textwidth]{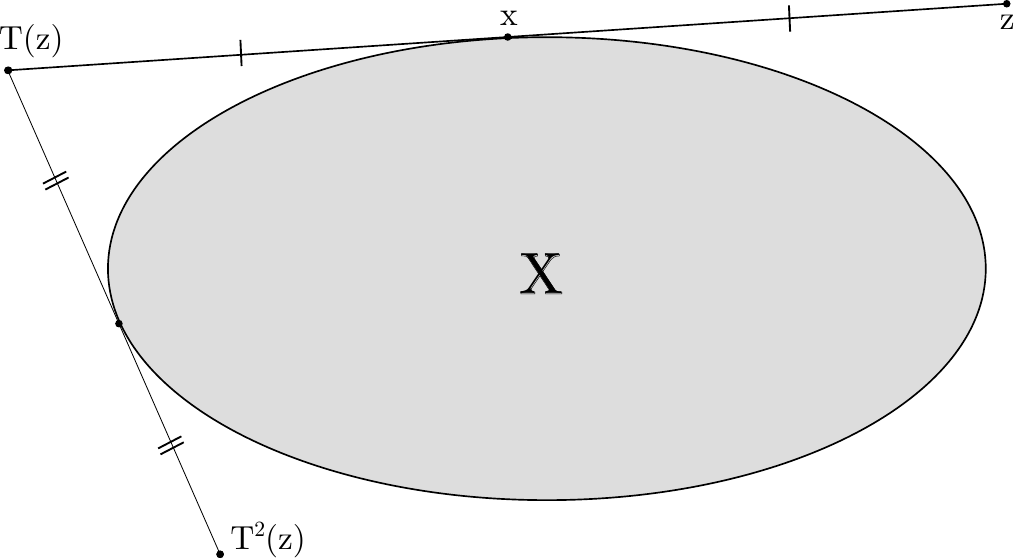}
	\caption{The symplectic outer billiard map in the plane}
	\label{fig:outer_billiard}
\end{figure}

Reversing the orientation of the characteristic line bundle, one gets the inverse map $T^{-1}$. Therefore, the map $T$ is a bijection of $\mathbb{R}^{2n}\setminus X$.

The map $T$ is only continuous in general, if we assume that $\partial X$ is $C^k$-smooth and has positive curvature, then $T$ is a $C^{k-1}$-smooth symplectomorphism.

The existence and non-existence of periodic orbits, and especially 4-periodic orbits are discussed in the recent paper \cite{albers2024outersymplecticbilliards}. Our contribution is given by the following two theorems.

\begin{theorem}\label{thm:main}
	Let $X \subset \mathbb{R}^{2n}$ be a symplectically self-polar convex body with a $C^1$-smooth boundary. Then for every $x \in \partial X$ the following four points
	\[
		z_1(x) = x + f(x),\  z_2(x) = -x + f(x),\  z_3(x) = -x -f(x),\  z_4(x) = x-f(x)
	\] 
	are vertices of a parallelogram of symplectic area 4, which	form a centrally-symmetric closed outer billiard trajectory, i.e. $Tz_i = z_{i+1}$, where $z_5 = z_1$.
\end{theorem}

\begin{proof}
	Let us show that $Tz_4 = z_1$. It was observed before that the vector $f(x)$ gives the characteristic direction at the point $x$. Therefore, the line $z_4z_1$ coincides with the characteristic line at the point $x$ and the point $x$ divides the segment $[z_4,z_1]$ in half. From the definition of the map $T$ we have $Tz_4 = z_1$.

	By Lemma~\ref{lem:f_compose} the characteristic direction at the point $y = f(x) \in \partial X$ is given by the vector $f(y) = f(f(x)) = -x$. Therefore, $Tz_1 = z_2$. The other cases follow from central symmetry.
\end{proof}

\begin{theorem}\label{thm:main_reverse}
	Let $X \subset \mathbb{R}^{2n}$ be a centrally symmetric strictly convex body with a $C^2$-smooth boundary. Suppose that for every point $x$ of $\partial X$ there exists centrally symmetric $4$-periodic trajectory of $T$ which passes through $x$ (i.e. $x = (z+Tz)/2$ for some point $z$ on this trajectory). Then $X = \alpha X^\omega$ for some $\alpha > 0$, so that $\frac{1}{\sqrt \alpha} X$ is symplectically self-polar.
\end{theorem}

\begin{proof}
	Consider some $x \in \partial X$. Take point $z$ on the corresponding centrally symmetric $4$-periodic trajectory such that $x = (z+Tz)/2$, and consider point $y = (Tz + T^2 z)/2 \in \partial X$. From the central symmetry of the trajectory we have $y = (Tz - z)/2 =x - z$, i.e. $y$ is parallel to the characteristic line at $x$, so now $y$ may not coincide with $f(x)$ but still co-directed, so $y = f(x)/\|f(x)\|_X$. For the same reason $x$ is parallel to the characteristic line at $y$.

	Therefore, the equality $\omega(x,\cdot) = \omega(x,y)$ determines the tangent hyperplane of $\partial X$ at $x$ and the equality $\omega(\cdot,y) = \omega(x,y)$ determines the tangent hyperplane of $\partial X$ at $y$.

	As a result, we have that for every $x \in \partial X$ the point
	\[
	\left(x,\frac{f(x)}{\|f(x)\|_X}\right)
	\]
	is critical for the function $\omega|_{\partial X \times \partial X}$. Note that in our setting $f(x)/\|f(x)\|_X$ smoothly depends on $x$.{} Therefore, \[\omega\left(x,\frac{f(x)}{\|f(x)\|_X}\right) = \alpha\] for every $x \in \partial X$ and some $\alpha > 0$. On the other hand, from the definition of $f(x)$ we have
	\[\omega\left(x,\frac{f(x)}{\|f(x)\|_X}\right) = 1/\|f(x)\|_X,\]
	i.e. $\|f(x)\|_X = 1/\alpha$ for every $x \in \partial X$. Using this, it is easy to show that $X = \alpha X^\omega$.
\end{proof}

\begin{remark}
	Note that if $X$ has only $C^1$-smooth boundary, this proof does not work. The set
	\[
		\left\{\left(x,\frac{f(x)}{\|f(x)\|_X}\right): x \in \partial X\right\}
	\]
	is not a $C^1$-smooth connected submanifold of $\partial X \times \partial X$ anymore. It is a path-connected subset of the critical points of $\omega$, but it is not enough to conclude  that $\omega|_{\partial X \times \partial X}$ is constant on this subset~\cite{Whitney}.
\end{remark}

\begin{remark}
	Theorem~\ref{thm:main} and Theorem~\ref{thm:main_reverse} can be formulated and proved in a similar way for symplectic billiards. For the definition and properties of symplectic billiards see~\cite{albers2018introducing}.
\end{remark}

\section{Invariant hypersurface}\label{sec:invariant_hypersurface}

 Throughout this section we assume $X \subset \mathbb{R}^{2n}$ to be a symplectically self-polar convex body. Theorem~\ref{thm:main} and Lemma~\ref{lem:f_compose} imply that for such an $X$ with a $C^1$-smooth boundary the set $Y = \{x + f(x): x \in \partial X\}$ is invariant under an outer billiard map. Indeed, $$T(x + f(x)) = -x + f(x) = f(x) + f(f(x)) \in Y.$$ In this section we discuss the geometric properties of $Y$.

\begin{lemma}\label{lem:top_sphere}
	The set $Y$ is a topological sphere, and $X$ lies in the bounded component of $\mathbb{R}^{2n} \setminus Y$. If $X$ has $C^k$-smooth boundary for $k \geq 2$, then $Y$ is a $C^{k-1}$-embedded submanifold.
\end{lemma}
\begin{proof}
	Consider the map $g \colon \partial X \to \mathbb{R}^{2n}$ defined as $g(x) = x+f(x)$. Notice that this map is injective, because the map $T$ is injective. Therefore, $Y$ is a topological sphere and from the Jordan--Brouwer separation theorem, $\mathbb{R}^{2n}\setminus Y$ has two connected components.

	Notice that $Y$ does not intersect $X$, since $X$ is strictly convex. Let us define the homotopy $h_t(x) = x + t\cdot f(x)$ for $0 \leq t \leq 1$ and $x \in \partial X$. Observe that $h_0 = \id_{\partial X}$, $h_1 = g$ and $h_t(x) \neq 0$. Therefore, the winding number of $g$ with respect to the origin coincides with the winding number of $\id_{\partial X}$ with respect to the origin which is obviously odd. It then follows that the origin belongs to the bounded component of $\mathbb{R}^{2n}\setminus Y$. This implies that $X$ lies in a bounded component of $\mathbb{R}^{2n}\setminus Y$.

	Assume now that $X$ has a $C^k$-smooth boundary for $k \geq 2$, we are going to show that the $C^{k-1}$-smooth map $g(x) = x + f(x)$ has non-degenerate differential at every point of the boundary $\partial X$. Let $\xi$ be a non-zero tangent vector to $\partial X$ at some point. If $g_*(\xi) = 0$, then  $\xi + \nabla_{\xi}f = 0$. Therefore, $0 = \omega(\xi, \xi + \nabla_{\xi}f) = \omega(\xi,\nabla_{\xi}f)$. But from Lemma~\ref{lem:strong_conv} we have $\omega(\xi,\nabla_{\xi}f) > 0$, contradiction. Thus, $Y$ is a $C^{k-1}$-smooth embedded submanifold.
\end{proof}

\begin{theorem}{\label{thm:starbody}} If $\partial X$ is $C^1$-smooth, then $Y$ is a boundary of a star-shaped (not necessarily convex) body. Moreover, if $\partial X$ is $C^2$-smooth, then every radial ray is transversal to $Y$.
\end{theorem}
\begin{proof}
	To prove that $Y$ is a boundary of a star-shaped body, we have to prove that every ray of the form $\{tv: t \in [0,+\infty)\}$ intersects $Y$ exactly at one point for every $v \in \mathbb{R}^{2n}\setminus\{0\}$. From Lemma~\ref{lem:top_sphere}, we know that for every $v \in \mathbb{R}^{2n}\setminus\{0\}$ the corresponding ray intersects $Y$ at least at one point. Assume now that  $y+f(y) = t(x + f(x))$ for some $x,y \in \partial X$ and $t \geq 1$. Consider the symplectic subspace $V = \mathop{\text{span}} \{x, f(x)\}$, then $\mathbb{R}^{2n} = V \oplus V^{\omega}$, where $V^{\omega} = \{u \in \mathbb{R}^{2n} :\, \omega(u,w) = 0\ \forall w \in V\}$. Let $y = y^\parallel + y^\perp$ and $f(y) = f(y)^\parallel + f(y)^\perp$ with respect to this decomposition. From $y+f(y) \in V$ we have $y^\perp + f(y)^\perp = 0$. Since $y^\parallel, f(y)^\parallel \in V$, we can write them as
	\begin{align*}
		y^\parallel    & = \alpha_1 x + \beta_1 f(x), \\
		f(y)^\parallel & = \alpha_2 x + \beta_2 f(x).
	\end{align*}
	Note that $|\alpha_i|,|\beta_i|\leq 1$. Indeed, $x,y,f(x),f(y) \in \partial X$ and $X$ is symplectically self-polar, therefore
	\begin{align*}
		|\alpha_1| & = |\omega(y^\parallel, f(x))| = |\omega(y,f(x))| \leq 1,       \\
		|\beta_1|  & = |\omega(y^\parallel, x)| = |\omega(y,x)| \leq 1,             \\
		|\alpha_2| & = |\omega(f(y)^\parallel, f(x))| = |\omega(f(y),f(x))| \leq 1, \\
		|\beta_2|  & = |\omega(f(y)^\parallel, x)| = |\omega(f(y),x)| \leq 1.
	\end{align*}

	Since $y+f(y) = t(x + f(x))$ for $t \geq 1$ we also have $\alpha_1 + \alpha_2 = \beta_1 + \beta_2 = t \geq 1$, therefore $0 \leq \alpha_i,\beta_i \leq 1$. Using $y^\perp + f(y)^\perp = 0$, we have $(\alpha_1\beta_2 - \alpha_2\beta_1) = \omega(y^\parallel, f(y)^\parallel) = \omega(y,f(y)) = 1$. Thus
	\[
		\begin{cases}
			\alpha_1\beta_2 - \alpha_2\beta_1 = 1,              \\
			\alpha_1 + \alpha_2 = \beta_1 + \beta_2 = t \geq 1, \\
			0 \leq \alpha_i,\beta_i \leq 1.
		\end{cases}
	\]
	It is easy to see that this system has a solution only if $t = 1$ and the solution is $(\alpha_1, \beta_1, \alpha_2, \beta_2) = (1,0,0,1)$. Therefore, $x + f(x) = y + f(y)$, meaning $x = y$.

	Now, let $\partial X$ be $C^2$-smooth. Notice that the tangent space of $Y$ at the point $x + f(x)$ for $x \in \partial X$ can be described as $\{\xi + \nabla_\xi f: \xi \in T_x\partial X\}$. Consider the symplectic subspaces $V  = \mathop{\text{span}} \{x,f(x)\}$ and $V^\omega$. Remember that $f(x)$ and $x$ give characteristic directions in $T_x \partial X$ and $T_{f(x)} \partial X$ respectively. Therefore, by the definition of characteristic direction we have
	\begin{align*}
		 & T_{x} \partial X = \langle f(x) \rangle^\omega = \langle f(x) \rangle \oplus V^\omega, \\
		 & T_{f(x)} \partial X = \langle x \rangle^\omega = \langle x \rangle \oplus V^\omega.
	\end{align*}
	Arguing by contradiction, assume that there exists $x \in \partial X$ and $\xi \in T_x \partial X$ such that
	$$\xi + \nabla_{\xi} f = x + f(x).$$
	Using this and the fact that $\xi \in T_{x} \partial X$, $\nabla_{\xi}f \in T_{f(x)} \partial X$, together with decomposition of these spaces one can obtain that
	\begin{align*}
		\xi          & = f(x) + \eta, \\
		\nabla_\xi f & = x - \eta,
	\end{align*}
	for some $\eta \in V^\omega$. Therefore,
	\[
		\omega(\xi,\nabla_{\xi}f) = \omega(f(x),x) = -1.
	\]
	But this contradicts Lemma~\ref{lem:strong_conv}.
\end{proof}

\begin{lemma}{\label{lem:segment}}
	Every line of the form $\{x + tf(x): t \in \mathbb{R}\}$ intersects $Y$ exactly at two points $x + f(x)$ and $x-f(x)$ for every $x \in \partial X$. If $\partial X$ is $C^2$-smooth, then every such line is transversal to $Y$.
\end{lemma}

\begin{proof}
	First, notice that the rays  $\{x + tf(x): t \in [0,+\infty)\}$ and $\{y + tf(y): t \in [0,+\infty)\}$ do not intersect for different $x,y \in \partial X$, because $T$ is an injective map, in particular $y + f(y) \notin \{x + tf(x): t \in [0,+\infty)\}$. In the same way, considering $T^{-1}$ instead of $T$ one can prove that $y - f(y) \notin \{x - tf(x): t \in [0,+\infty)\}$ for different $x,y \in \partial X$. Combine it with $Y = \{y + f(y): y \in \partial X\} = \{y - f(y): y \in \partial X\}$, we obtain the desired result.

	Now, let $\partial X$ be $C^2$-smooth. Arguing by contradiction, assume that the line $\{x + tf(x): t \in \mathbb{R}\}$ intersects $Y$ not transversally at the point $x + f(x)$ for $x \in \partial X$ (case $x - f(x)$ is analogous). Following the proof of Theorem~\ref{thm:starbody}, there exists $\xi \in T_{x} \partial X$ such that $$\xi + \nabla_\xi f = f(x),$$ which implies
	\begin{align*}
		\xi = f(x) + \eta, \\
		\nabla_\xi f = -\eta,
	\end{align*}
	for $\eta \in V^\omega$, where $V = \langle x , f(x)\rangle$. Therefore, $\omega(\xi, \nabla_{\xi}f) = 0$, contradiction with Lemma~\ref{lem:strong_conv}.
\end{proof}

\begin{remark}
	For a $C^2$-smooth planar curve of positive curvature, the outer billiard map is a twist symplectic map of the phase cylinder, with respect to two sets of symplectic coordinates. The first set is the so-called envelope coordinates \cite{MR1419453}, and the second one is the symplectic polar coordinates \cite{MR4722194}. It then follows from Birkhoff theorem,  that any invariant curve must be a graph of a Lipschitz function with respect to the vertical foliations of these coordinate systems. This argument implies Theorem~\ref{lem:segment} and Lemma~\ref{thm:starbody} for the $C^2$-smooth planar case. To the best of our knowledge, it is not known if there are reminiscences of the twist condition and Birkhoff theorems for higher-dimensional outer billiards.
\end{remark}

The next theorem and the corollary show that for non-trivial symplectic self-polar bodies, one cannot always move the parallelogram of 4-periodic orbits in such a way that the tangency points move along the characteristics of $X$.

However, one can move the parallelogram of 4-periodic orbit in such a way that all vertices move along the characteristics of $Y$ (here we assume that $\partial X$ is $C^2$-smooth). This is because $Y$ is invariant under $T$, which is a symplectomorphism and hence maps characteristics to characteristics.

\begin{theorem}{\label {thm:planar}} Let $X\subset\mathbb R^{2n}$ be a $C^2$-smooth symplectically self-polar convex body. Let  $\gamma(t)$ be a characteristic curve on $\partial X$. If $\delta(t) = f(\gamma(t))$ is also a characteristic curve on $\partial X$, then  $\gamma$ and $\delta$ are planar curves.
\end{theorem}
\begin{proof}
	Since the characteristic line at $\gamma(t)$ is parallel to $f(\gamma(t))=\delta(t)$, and the characteristic line at $\delta(t)$ is parallel to $f(\delta(t))=-\gamma(t)$, then
	\[
		\begin{cases}
			\dot{\gamma}(t) = \alpha(t) \cdot \delta(t), \\
			\dot{\delta}(t) = \beta(t) \cdot \gamma(t),
		\end{cases}
	\]
	for some continuous functions $\alpha(t)$ and $\beta(t)$. These two equations imply that $\gamma(t), \delta(t)$ are planar curves.
\end{proof}
Using Theorem \ref{thm:planar} and  the result of ~\cite[Theorem 1.3]{karasev2024convex} (which is stated in the $C^\infty$ category) we get the following:
\begin{corollary}\label{cor:planar}Let $X\subset\mathbb R^{2n}$ be a $C^\infty$-smooth symplectically self-polar convex body.
	If $f$ maps the characteristics of $\partial X$ to the characteristics of $\partial X$, then $X$ is linearly symplectomorphic to the unit Euclidean ball.
\end{corollary}
Similarly, we have the following:
\begin{corollary}
	Let $X\subset\mathbb R^{2n}$ be a $C^\infty$-smooth symplectically self-polar convex body.
	Let $\psi_\pm \colon \partial X\to Y$ be the map defined by
	$\psi_\pm(x) = x \pm f(x)$. If $\psi_\pm$ maps the characteristics of $\partial X$ to the characteristics of $Y$, then $X$ is linearly symplectomorphic to the unit Euclidean ball.
\end{corollary}
\begin{proof}
	Consider any characteristic curve $\gamma(t)$ on the boundary $\partial X$,  and the corresponding curve $\sigma_\pm (t) = \gamma(t) \pm f(\gamma(t))$. By the assumptions, $\sigma_\pm$ is a characteristic curve on $Y$. Hence, we know that $\psi_\pm$ are diffeomorphisms and the differentials of $\psi_\pm$ map the characteristic line of $X$ at $x$ to the characteristic lines of $Y$ at $\psi_\pm(x)$ for every $x \in \partial X$.  Therefore, the differential of the map $f=\psi_-^{-1} \circ \psi_+ \colon \partial X \to\partial  X$ maps the characteristic line of $\partial X$ at $x$ to the characteristic line of $\partial X$ at $f(x)$. Hence, Corollary \ref{cor:planar} implies the result.
\end{proof}

\section{Examples}\label{sec:examples}

In this section, we discuss how to construct symplectically self-polar convex bodies with $C^1$-smooth boundary and show that there exist symplectically self-polar convex bodies with a $C^\infty$-smooth boundary that are not linearly symplectomorphic to the unit Euclidean ball.

\textbf{$C^1$-smooth boundary:} The following two examples were introduced in~\cite{berezovik2022symplectic}.

First, introduce the $l_2$-sum operation for centrally symmetric convex bodies. Let $X \subset \mathbb{R}^n$ and $Y \subset \mathbb{R}^m$ be centrally symmetric convex bodies, then their $l_2$-sum is the following set $$X \oplus_2 Y = \{(x,y) \in \mathbb{R}^n \times \mathbb{R}^m: \|x\|_X^2 + \|y\|_Y^2 \leq 1\}.$$
\begin{itemize}[leftmargin=*]
	\item Let $K \subset \mathbb{R}^n$ be a centrally symmetric strictly convex body with $C^1$-smooth boundary, then $K^\circ$ is also strictly convex and has $C^1$-smooth boundary. Consider their Lagrangian $l_2$-sum $X = K \oplus_2 K^\circ \subset \mathbb{R}^n \times \mathbb{R}^n$. Note that $X$ is symplectically self-polar, indeed
	      \[
		      X^\omega = J(K \oplus_2 K^\circ)^\circ = J(K^\circ \oplus_2 K) = K \oplus_2 K^\circ = X.
	      \]
	      Here we used that $l_2$-sum commutes with polar transformation. Moreover, it is easy to check that $X$ is strictly convex, then by Lemma~\ref{lem:strict_conv} $X$ has $C^1$-smooth boundary (this is also true because an $l_2$-sum of bodies with a $C^1$-smooth boundary has a $C^1$-smooth boundary).
	\item Let $X \subset \mathbb{R}^{2n}$ and $Y \subset \mathbb{R}^{2m}$ be symplectically self-polar convex bodies with $C^1$-smooth boundary. Consider their symplectic $l_2$-sum $X \oplus_2 Y \subset \mathbb{R}^{2n} \times \mathbb{R}^{2m}$, then it is again symplectically self-polar because a symplectic $l_2$-sum commutes with a symplectic polarity. $X \oplus_2 Y$ is again strictly convex and has a $C^1$-smooth boundary.
\end{itemize}
Note that these procedures do not allow us to obtain symplectically self-polar bodies with a $C^2$-smooth boundary, because in general the case an $l_2$-sum of bodies with $C^\infty$-smooth boundary does not have $C^2$-smooth boundary.

\textbf{$C^\infty$-smooth boundary:} Take an arbitrary $\varepsilon > 0$ and consider the set $U_{\varepsilon} \subset \mathbb{R}^{2n}$ of points whose coordinates are all greater than $\varepsilon$. Let $X \subset \mathbb{R}^{2n}$ be a centrally symmetric convex body with a $C^\infty$-smooth boundary and a positive curvature, such that $X$ coincides with the unit Euclidean ball $B \subset \mathbb{R}^{2n}$ outside $U_{\varepsilon} \cup (-U_{\varepsilon})$. Then $X^\circ$ is also a centrally symmetric convex body with a $C^\infty$-smooth boundary and a positive curvature, and $X^\circ$ coincides with $B$ outside $U_{\varepsilon} \cup (-U_{\varepsilon})$. Now consider the centrally symmetric convex body $Z \subset \mathbb{R}^{2n}$ which coincides with the $X$ in $U_{\varepsilon} \cup (-U_{\varepsilon})$ and coincides with $JX^\circ$ in $JU_{\varepsilon} \cup (-JU_{\varepsilon})$ and coincides with the ball $B$ outside all these sets. One can check that this body has $C^\infty$-smooth boundary and it is symplectically self-polar because $Z = JZ^\circ = Z^\omega$.
\section{Discussion and open questions}

\begin{question}
	Are there two different symplectically self-polar convex bodies $X_1,X_2 \subset \mathbb{R}^{2n}$ with a smooth boundary such that the corresponding $Y_1$ and $Y_2$ coincide?
\end{question}

In dimension two, for bodies with a $C^2$-smooth boundary, the answer is negative. Indeed, consider a symplectically self-polar convex body $X \subset \mathbb{R}^2$ and the corresponding set $Y$. Then $\partial X$ can be recovered from $Y$ by the so-called area construction (the outer billiard analog of the string construction \cite{tabachnikov1995dual}) as follows.
Let $A$ be the area of the region bounded by $Y$. For any point $y$,
consider the segment $[y,z(y)], z(y)\in Y$ cutting from $Y$ a segment of a constant area equal to $\frac{A-4}{4}$ (Figure~\ref{fig_1}, gray area). Then,
\[
	\partial X = \left\{\frac{y + z(y)}{2}: y \in Y\right\}.
\]

\begin{figure}[ht]
	\centering
	\includegraphics[width=0.8\textwidth]{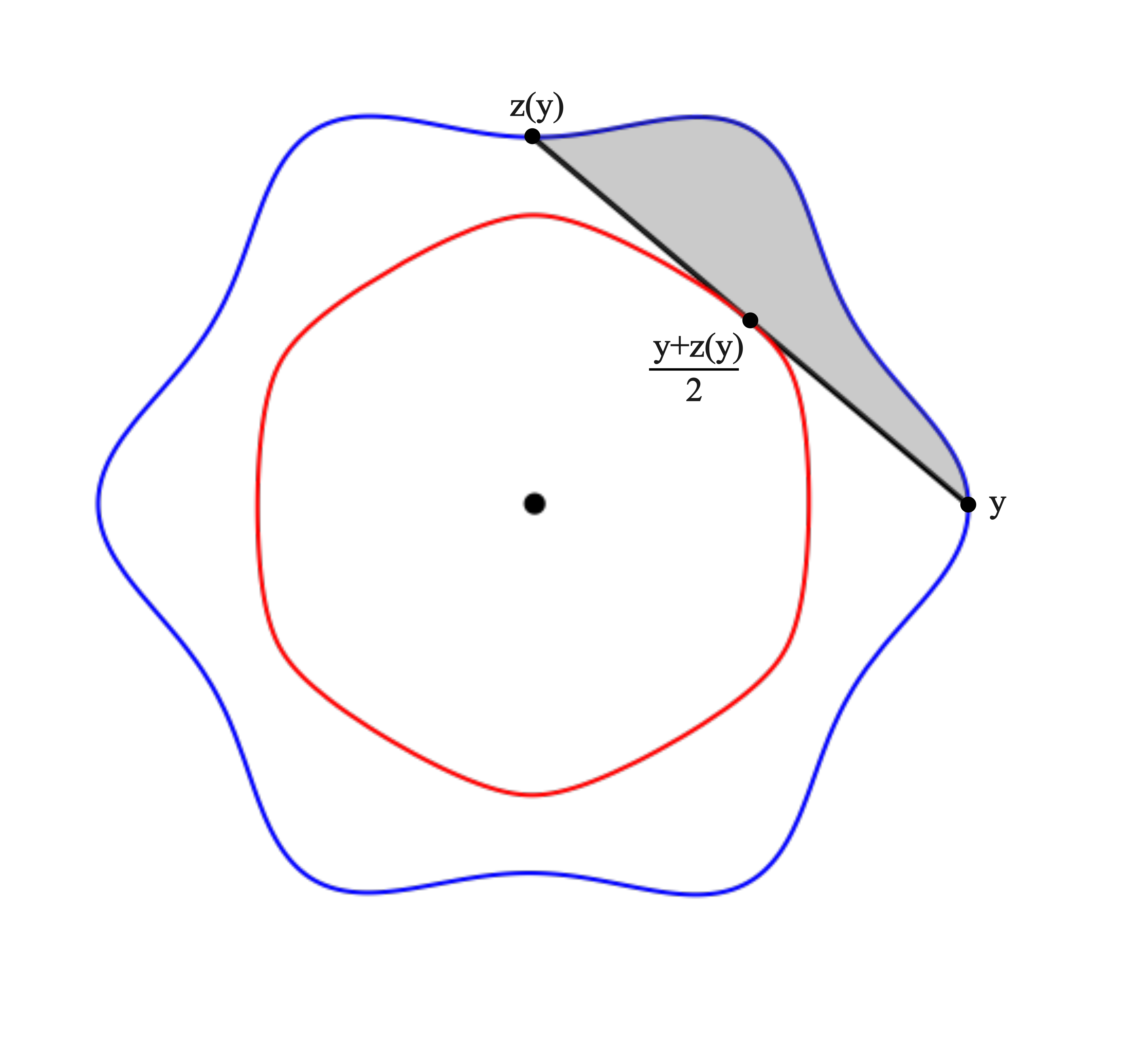}
	\caption{{Area construction applied to $Y$ with the area parameter $\frac{A-4}{4}$ }.}
	\label{fig_1}
\end{figure}

\begin{question}
	What does $Y$ remember about $X$? Is there some non-trivial relation between the volume which is bounded by $Y$ and the volume of $X$, or maybe between some other invariants (e.g. capacity, systole)? In dimension two, one can prove that the volume bounded by $Y$ equals twice the volume of~$X$.
\end{question}

\begin{question}
	It is not clear if the result of Theorem \ref{thm:main_reverse} remains true for a $C^1$-smooth convex symplectically self-polar body $X$. In the $C^1$-case, the map $f$ is only continuous and we cannot use partial derivatives.
\end{question}

\begin{question}
	In \cite{bialy2022self}, analytic examples of Radon curves were found via a Lam\'e equation. Interestingly, are there analytic examples of symplectically self-polar bodies in higher dimensions?

\end{question}

\begin{question}
	We can generalize the set $Y$ for an arbitrary symplectically self-polar convex body $X$, not necessarily with a smooth boundary, as
	\[
		Y = \{x+y:\ x,y \in \partial X,\: \omega(x,y)  = 1\}.
	\]
	What are the geometric or topological properties of $Y$ in this case? Is it still a topological sphere?
\end{question}
\begin{question}
	For planar sufficiently smooth outer billiards, KAM theory ensures the existence of invariant curves at infinity and near the boundary. It then follows in particular that all orbits are bounded.
	A general question on outer billiards in higher dimensions arises: if there exist unbounded orbits.

\end{question}

\bibliography{bibliography}

\begin{thebibliography}{10}

\bibitem{albers2024outersymplecticbilliards}
P.~Albers, A.~C. Caliz, and S.~Tabachnikov.
\newblock Outer symplectic billiards, 2024.
\newblock \href{https://arxiv.org/abs/2409.07990}{arXiv:2409.07990}.

\bibitem{albers2018introducing}
P.~Albers and S.~Tabachnikov.
\newblock Introducing symplectic billiards.
\newblock {\em Advances in Mathematics}, 333:822--867, 2018.

\bibitem{berezovik2023symplectically}
M.~Berezovik.
\newblock Symplectically self-polar polytopes of minimal capacity.
\newblock {\em Annales math{\'e}matiques du Qu{\'e}bec}, 49(2):335--353, 2025.

\bibitem{berezovik2022symplectic}
M.~Berezovik and R.~Karasev.
\newblock Symplectic polarity and {M}ahler's conjecture.
\newblock {\em Israel Journal of Mathematics}, to appear.
\newblock \href{https://arxiv.org/abs/2211.14630}{arXiv:2211.14630}.

\bibitem{MR1071638}
M.~Berger.
\newblock Sur les caustiques de surfaces en dimension {$3$}.
\newblock {\em C. R. Acad. Sci. Paris S\'er. I Math.}, 311(6):333--336, 1990.

\bibitem{MR4722194}
M.~Bialy.
\newblock Integrable outer billiards and rigidity.
\newblock {\em J. Mod. Dyn.}, 20:51--65, 2024.

\bibitem{bialy2022self}
M.~Bialy, G.~Bor, and S.~Tabachnikov.
\newblock Self-b{\"a}cklund curves in centroaffine geometry and {L}am{\'e}’s
  equation.
\newblock {\em Communications of the American Mathematical Society},
  2(06):232--282, 2022.

\bibitem{MR1419453}
P.~Boyland.
\newblock Dual billiards, twist maps and impact oscillators.
\newblock {\em Nonlinearity}, 9(6):1411--1438, 1996.

\bibitem{Day}
M.~M. Day.
\newblock Some characterizations of inner-product spaces.
\newblock {\em Transactions of the American Mathematical Society},
  62(2):320--337, 1947.

\bibitem{MR1348796}
P.~M. Gruber.
\newblock Only ellipsoids have caustics.
\newblock {\em Math. Ann.}, 303(2):185--194, 1995.

\bibitem{karasev2024convex}
R.~Karasev and A.~Sharipova.
\newblock Convex bodies with all characteristics planar.
\newblock {\em International Mathematics Research Notices},
  2024(10):8104--8121, 2024.

\bibitem{Mahler1939}
K.~Mahler.
\newblock Ein Übertragungsprinzip für konvexe körper.
\newblock {\em Časopis pro pěstování matematiky a fysiky},
  068(3-4):93--102, 1939.

\bibitem{martini2006antinorms}
H.~Martini and K.~J. Swanepoel.
\newblock Antinorms and {R}adon curves.
\newblock {\em Aequationes mathematicae}, 72:110--138, 2006.

\bibitem{Radoncurves}
J.~Radon.
\newblock {ü}ber eine besondere {A}rt ebener {K}urven.
\newblock {\em Berichte der {S}achsische {A}kademie der {W}issenschaften zu
  Leipzig}, 68:131--134, 1916.

\bibitem{tabachnikov1995billiards}
S.~Tabachnikov.
\newblock {\em Billiards}.
\newblock Panoramas et synth{\`e}ses - Soci{\'e}t{\'e} math{\'e}matique de
  France. Soci{\'e}t{\'e} math{\'e}matique de France, 1995.

\bibitem{tabachnikov1995dual}
S.~Tabachnikov.
\newblock On the dual billiard problem.
\newblock {\em Advances in Mathematics}, 115(2):221--249, 1995.

\bibitem{Whitney}
H.~Whitney.
\newblock {A function not constant on a connected set of critical points}.
\newblock {\em Duke Mathematical Journal}, 1(4):514 -- 517, 1935.

\end{thebibliography}
\bibliographystyle{abbrv}

\end{document}